\documentclass[11pt,oneside, reqno]{amsart}

\usepackage[colorlinks,linkcolor=blue,citecolor=blue]{hyperref}
\usepackage{graphics}
\usepackage{enumitem}
\usepackage{marginnote}	
\usepackage{latexsym, amssymb, amsmath, amsthm}
\usepackage{mathrsfs}
\usepackage{tikz}
\usepackage{tikz-cd}
\usepackage[only,llbracket,rrbracket]{stmaryrd}
\usetikzlibrary{decorations.pathmorphing}

\tikzset{
	squiggly/.style={decorate, decoration={snake, amplitude=1mm, segment length=6mm}},
}

\newtheorem{theorem}{Theorem}[section]
\newtheorem{lemma}[theorem]{Lemma}
\newtheorem{proposition}[theorem]{Proposition}
\newtheorem{corollary}[theorem]{Corollary}

\theoremstyle{definition}\newtheorem{definition}[theorem]{Definition}
\theoremstyle{definition}
\theoremstyle{definition}\newtheorem{remark}[theorem]{Remark}

\def\ga{\mathfrak{a}}
\def\gb{\mathfrak{b}}

\def\gp{\mathfrak{p}}

\def \sl{\mathfrak{sl}}

\def \CB{\mathcal{B}}

\def \CD{\mathcal{D}}

\def \CO{\mathcal{O}}

\def \CS{\mathcal{S}}

\def \CY{\mathcal{Y}}

\def \Prim{{\rm Prim}}
\def\Max{{\rm Max}}

\def \Aut{{\rm Aut}}

\def \Spec{{\rm Spec}}
\def \ann{{\rm ann}}
\def \Frac{{\rm Frac}}

\def \char{\operatorname{char}}

\def \mK{\Bbbk}
\def \Z{\mathbb{Z}}
\def \N{\mathbb{N}}

\def \D{\Delta}
\def \d{\delta}
\def \s{\sigma }
\def \l{\lambda}

\begin{document}

	\author{Tao Lu}
	
	
	
	\address{School of Mathematical Science, Yangzhou University, Yangzhou, 225002, China}
	
	\email{taolu@yzu.edu.cn}
	
	\begin{abstract} 
We study a family of Hopf algebras arising as liftings of the Jordan plane over the infinite cyclic group. We determine their centres, prime and primitive spectra, and automorphism groups. We show that every prime ideal is completely prime and that every nonzero ideal intersects the centre nontrivially. We construct explicit simple modules corresponding to all primitive ideals and classify the finite-dimensional simple modules. Finally, we prove that these Hopf algebras satisfy the Dixmier--Moeglin equivalence.
	\end{abstract}
	
	\title{On a family of liftings of the Jordan plane}
	\maketitle
	

\section{Introduction}

 Throughout the paper, all modules are left modules. The base field $\mK$ is assumed to be algebraically closed of characteristic zero. We write $\mK^*=\mK\setminus{0}$ and $\mathbb{N}=\{0,1,2,\dots\}$.

The lifting method of Andruskiewitsch and Schneider \cite{Andruskiewitsch-Schneider-1,Andruskiewitsch-Schneider-2} is a fundamental tool in the classification of pointed Hopf algebras. Given the bosonization of a Nichols algebra, one seeks all Hopf algebras whose associated graded Hopf algebra with respect to the coradical filtration coincides with that bosonization. Such Hopf algebras are called \emph{liftings}, and can be viewed as deformations of the algebra structure that preserve the underlying coalgebra structure. This method has led to many classification results for pointed Hopf algebras;  see, for example, \cite{Andruskiewitsch-Angiono-Garcia Iglesias,Andruskiewitsch-Schneider-3,Angiono-Garcia Iglesias} and the references therein. 

The Jordan plane is the quadratic algebra $\mathtt{J}=\Bbbk\langle x,y \mid yx-xy=-\tfrac12 x^2\rangle$.
It is a well-known Artin--Schelter regular algebra of global dimension two \cite{Artin-Shelter}, and one of the fundamental examples of a Nichols algebra of non-diagonal type with finite Gelfand--Kirillov dimension \cite{Andruskiewitsch-Angiono-Heckenberger}. Its algebraic and homological properties have been extensively studied; see \cite{Artin-Shelter,Iyudu,Shirikov}. Liftings of the Jordan plane and the super Jordan plane over nilpotent-by-finite groups were computed in \cite{Andruskiewitsch-Angiono-Heckenberger-1,Andruskiewitsch-Angiono-Heckenberger-2}.

In this paper, we study a family of Hopf algebras $\mathfrak{U}(\lambda)$ arising as liftings of the Jordan plane over the group algebra $\Bbbk G$, where $G$ is the infinite cyclic group. According to \cite[Proposition 1]{Andruskiewitsch-Angiono-Heckenberger-2}, the liftings of the Jordan plane over $\Bbbk G$ fall into three families. The first family consists of the Hopf algebras $\mathfrak{U}(\lambda)$ considered here, while the second family is defined in \cite[Definition 9]{Andruskiewitsch-Angiono-Heckenberger-2}. The third family is the Jordanian lifting $\mathfrak{U}^{\mathtt{jordan}}$, which is closely related to the Jordanian quantum algebra $U_h(\mathfrak{sl}_2)$ introduced by Ohn \cite{Ohn}. It was shown in \cite{Lu} that $\mathfrak{U}^{\mathtt{jordan}}$ is isomorphic to a localization of $U(\mathfrak{sl}_2)$, and its prime spectrum and finite-dimensional simple modules have been determined.

Our aim is to investigate the structure and representation theory of the Hopf algebras $\mathfrak{U}(\lambda)$. We prove that the centre of $\mathfrak{U}(\lambda)$ is the polynomial algebra $\Bbbk[\Omega]$. We determine the prime, primitive, maximal, and completely prime spectra, classify the finite-dimensional simple modules, and compute the automorphism groups. We also show that $\mathfrak{U}(\lambda)$ satisfies the Dixmier--Moeglin equivalence for every $\lambda\in\Bbbk$. The cases $\lambda=0$ and $\lambda\neq0$ exhibit substantially different behavior and are treated separately.

We begin with $\mathfrak{U}(0)$, the bosonization of the Jordan plane. We give an explicit description of its prime and primitive spectra, showing that every primitive ideal is maximal and every prime ideal is completely prime. We further prove that every nonzero ideal intersects the centre nontrivially. As consequences, we classify the finite-dimensional simple modules and reduce the classification of infinite-dimensional simple modules to that of a localization of the first Weyl algebra. Finally, we determine the automorphism group of $\mathfrak{U}(0)$ using the structure of the centre, the height-one prime ideals, and the centralizer of the element $g$.

We then turn to the non-Jordanian liftings $\mathfrak{U}(\lambda)$ with $\lambda\neq0$. A crucial step in this setting is the analysis of the central quotients $\mathfrak{U}^{\alpha}=\mathfrak{U}(\l)/\langle \Omega-\alpha \rangle$ for $\alpha \in \mK$. We prove that for generic values $\alpha\neq \pm 4\lambda$, the quotient $\mathfrak{U}^{\alpha}$ is a simple domain, whereas the exceptional quotients $\mathfrak{U}^{\pm 4\lambda}$ are neither simple nor domains. This dichotomy governs much of the ideal-theoretic and representation-theoretic structure of $\mathfrak{U}(\lambda)$. Using these results, we determine the prime, primitive, maximal, and completely prime spectra. As in the case $\lambda=0$, every prime ideal of $\mathfrak{U}(\lambda)$ is completely prime, and every nonzero ideal intersects the centre nontrivially. However, unlike the case $\lambda=0$, not all primitive ideals are maximal. We also construct explicit simple modules whose annihilators realize all primitive ideals. Furthermore, we determine the automorphism group of $\mathfrak{U}(\lambda)$ and show that it is smaller than that of $\mathfrak{U}(0)$, reflecting the increased rigidity of the nontrivial liftings. Finally, we prove that $\mathfrak{U}(\lambda)$ satisfies the Dixmier--Moeglin equivalence.

The paper is organized as follows. In Section \ref{Preliminaries}, we introduce the family of Hopf algebras $\mathfrak{U}(\l)$, construct a central element, and determine the centralizer of $g$. Section \ref{AlgJ} is devoted to the bosonization of the Jordan plane (the case $\mathfrak{U}(0)\cong J$), where we determine its centre, describe its prime and primitive spectra, and compute its automorphism group. In Section \ref{AlgU}, we study the non-Jordanian lifting $\mathfrak{U}(\l)$ with $\l \neq 0$. We determine its centre, describe its prime and primitive spectra, construct simple modules, and compute its automorphism group.

\section{Preliminaries} \label{Preliminaries}  

In this section, we introduce the family of Hopf algebras $\mathfrak{U}(\lambda)$ and establish some of their basic  properties. In particular, we construct a central element and determine the centralizer of the grouplike element $g$.

\subsection{The Hopf algebras $\mathfrak{U}(\l)$}
Let $G=\langle g \rangle$ be the infinite cyclic group generated by $g$, and let $V=\mK\{x, y \}\in {}_{\mK G}^{\mK G}\!\mathcal{YD}$ be the Yetter--Drinfeld module defined by 
\begin{equation*}
g\cdot x=x, \quad g\cdot y=y+x, \quad \d(x)=g\otimes x, \quad \d(y)=g\otimes y. 
\end{equation*}
By \cite[Proposition 3.4]{Andruskiewitsch-Angiono-Heckenberger}, the Nichols algebra $\CB(V)$ associated to $V$ turns out to be the Jordan plane. It is generated by $x$ and $y$ subject to the relation $yx-xy+\frac{1}{2}x^2=0$. Moreover, $\CB(V)$ has Gelfand--Kirillov dimension 2, and  $\{ x^i y^j \mid i,j\in \N \}$ forms a basis of $\CB(V)$. 
Since $\CB(V)$ is a Hopf algebra in the category ${}_{\mK G}^{\mK G}\!\CY\CD$, its Radford--Majid bosonization $J:=\CB(V)\#\mK G$ is a Hopf algebra in the usual category of $\mK$-vector spaces.
\begin{definition}
	The bosonization $J=\CB(V) \#\mK G$ of the Jordan plane is the associative $\mK$-algebra generated by $g^{\pm 1}$, $x$, and $y$ with defining relations
	\begin{equation*}
	\begin{gathered}
	gg^{-1}=g^{-1}g=1, \quad gxg^{-1}=x, \quad gyg^{-1}=y+x,\\
	yx-xy=-\tfrac{1}{2}x^2. 
	\end{gathered}	
	\end{equation*}
\end{definition}

The Hopf algebra $J$ has Gelfand--Kirillov dimension 3, and $\{ g^ix^j y^k \mid i\in \Z, j,k\in \N \}$ forms a basis of $J$. 
The main object of this paper is a family of liftings of $J$ obtained in \cite{Andruskiewitsch-Angiono-Heckenberger-1,Andruskiewitsch-Angiono-Heckenberger-2}.

\begin{definition} 
	Let $\l \in \mK$. The algebra $\mathfrak U(\lambda)$ is generated by $g^{\pm 1}$, $x$, and $y$ subject to the relations
	\begin{equation*}
	\begin{gathered}
	gg^{-1}=g^{-1}g=1, \quad gxg^{-1}=x, \quad gyg^{-1}=y+x,\\
	yx-xy=-\tfrac{1}{2}x^2+\l(1-g^2). 
	\end{gathered}	
	\end{equation*}
	Moreover, $\mathfrak{U}(\l)$ is a Hopf algebra with comultiplication, counit, and antipode given by 
	\begin{equation*}
	\begin{gathered}
	\D(g^{\pm 1})=g^{\pm 1} \otimes g^{\pm 1}, \quad \D(x)=x\otimes 1+ g\otimes x, \quad \D(y)=y\otimes 1+ g\otimes y,\\
	\varepsilon(g)=1, \quad \varepsilon(x)=0, \quad \varepsilon(y)=0,\\
	S(g)=g^{-1}, \quad S(x)=-g^{-1}x, \quad S(y)=-g^{-1}y. 
	\end{gathered}
	\end{equation*}
\end{definition}

In particular, $\mathfrak{U}(0) \cong J$. By \cite[Proposition 4.2]{Andruskiewitsch-Angiono-Heckenberger-1}, $\mathfrak{U}(\l)$ is a pointed Hopf algebra, and a cocycle deformation of its associated graded Hopf algebra $\operatorname{gr} \mathfrak{U}(\l) \cong J$. Moreover, $\mathfrak{U}(\l) \cong \mathfrak{U}(\l')$ if and only if there exists $c\in \mK^*$ such that $\l=c\l'$.

\subsection{A central element and the centralizer of $g$}
For every $\l \in \mK$, the algebra $\mathfrak{U}(\l)$ can be presented as an Ore extension
\begin{equation} \label{OreJ} 
\mathfrak{U}(\l)=\mK[g^{\pm 1},x][y;\d]
\end{equation}
where $\d$ is the derivation of $\mK[g^{\pm 1},x]$ determined by $\d(g)=-xg$ and $\d(x)=-\tfrac{1}{2}x^2+\l(1-g^2)$. It follows that $\mathfrak{U}(\l)$ is a Noetherian domain of Gelfand--Kirillov dimension 3. 

Define
\begin{equation} \label{Omega}  
\Omega:=x^2g^{-1}-2\l(g^{-1}+g). 
\end{equation}
Clearly, $\Omega$ commutes with both $g$ and $x$. Moreover, using the identities $[y,g]=-xg$, $[y,g^{-1}]=xg^{-1}$, and $[y,x^2]=-x^3+2\l(1-g^2)x$, we obtain
\begin{equation*}
\begin{aligned}\
[y, \Omega]&=[y,x^2]g^{-1}+x^2[y,g^{-1}]-2\l[y, g^{-1}+g]\\
&=\big(-x^3+2\l(1-g^2)x\big)g^{-1}+x^3g^{-1}-2\l(xg^{-1}-xg)\\
&=2\l(g^{-1}-g)x-2\l(xg^{-1}-xg)\\
&=0.
\end{aligned}
\end{equation*}
Hence $\Omega$ commutes with $y$, and therefore $\Omega$ is a central element of $\mathfrak{U}(\lambda)$.

Recall that for an algebra $R$ and $a\in R$, the centralizer of $a$ in $R$ is defined as $C_R(a):=\{ r\in R \mid ar=ra \}$. We next determine the centralizer of the element $g$ in $\mathfrak{U}(\l)$. 
\begin{lemma}  \label{a10Jun26}  
	For any $\l \in \mK$, one has $C_{\mathfrak{U}(\l)}(g)=\mK[g^{\pm 1}, x]$. 
\end{lemma}
\begin{proof}
	The inclusion $\mK[g^{\pm1},x]\subseteq C_{\mathfrak{U}(\lambda)}(g)$ is immediate. Conversely, 
	let $a=\sum_{i=0}^n r_i y^i\in C_{\mathfrak{U}(\l)}(g)$ where $r_i \in \mK[g^{\pm 1},x]$ and $r_n \neq 0$. Using $gyg^{-1}=y+x$, we obtain
	\begin{equation*}
	0=gag^{-1}-a=\sum_{i=0}^n r_i \big((y+x)^i-y^i\big).
	\end{equation*}
	For each $i\geq 1$, the element $(y+x)^i-y^i$ has $y$-degree $i-1$ with leading term $ixy^{i-1}$. Consequently, the right-hand side has $y$-degree $n-1$ and leading term  $nr_n xy^{n-1}$. Since $\mathfrak{U}(\lambda)$ is a domain and $r_n\neq0$, this term is nonzero whenever $n>0$. Therefore $n=0$, and hence $a=r_0\in \mK[g^{\pm 1}, x]$. This proves the reverse inclusion $C_{\mathfrak{U}(\l)}(g) \subseteq \mK[g^{\pm 1},x]$, and the result follows. 
\end{proof}

\section{Prime ideals, representations, and automorphisms of $J$} \label{AlgJ} 
In this section, we study the bosonization $J\cong \mathfrak{U}(0)$ of the Jordan plane. We determine its centre, describe its prime and primitive spectra, and compute its automorphism group.

\subsection{The centre of $J$}
The centre of a ring $R$ is denoted by $Z(R)$. Recall that an element $a\in R$ is said to be \emph{normal} if $aR=Ra$. Since $xg=gx$ and $xy=(y+\tfrac{1}{2}x)x$, 
the element $x$ is normal in $J$. When $\lambda=0$, the central element defined in \eqref{Omega} reduces to $\Omega=x^2g^{-1}\in Z(J)$. 

\begin{lemma}
	The centre of $J$ is the polynomial algebra $Z(J)=\mK[\Omega]$. 
\end{lemma}
\begin{proof}
	Let $J_x$ denote the localization of $J$ at the powers of the normal element $x$. and set $y':=yx^{-1}$. A straightforward computation yields
	\begin{equation*}
	[y',g]=-g, \qquad [y', x]=-\tfrac{1}{2}x.
	\end{equation*}	
	Since $\Omega=x^2g^{-1}$ is central and invertible in $J_x$, we may write $g^{\mp 1}=(\Omega x^{-2})^{\pm 1}$. It follows that $J_x$ is generated by $\Omega^{\pm 1}$, $x^{\pm 1}$, and $y'$. Thus
	\begin{equation} \label{J0x} 
	J_x \cong \mK[\Omega^{\pm 1}] \otimes \mK[x^{\pm 1}][y';\d],
	\end{equation}
	where $\d$ is the derivation of $\mK[x^{\pm 1}]$ determined by $\d(x)=-\tfrac{1}{2}x$. Since $\char \mK=0$, the Ore extension $\mK[x^{\pm1}][y';\delta]$ is simple with centre $\mK$. It follows that $Z(J_x)=\mK[\Omega^{\pm 1}]$. Consequently, we have $Z(J)=J\cap Z(J_x)=J \cap \mK[\Omega^{\pm 1}]=\mK[\Omega]$, which completes the proof.
\end{proof}

\subsection{Prime and primitive ideals of $J$}
Recall that a proper ideal $\gp$ of a ring $R$ is \emph{prime} if,  for all ideals $\ga$ and $\gb$ of $R$, the inclusion $\ga \gb \subseteq \gp$ implies $\ga \subseteq \gp$ or $\gb \subseteq \gp$. The set of prime ideals of $R$ is called the prime spectrum of $R$, and is denoted by $\Spec(R)$. 
A prime ideal $\gp$ is \emph{completely prime} if $R/\gp$ is a domain. For an $R$-module $M$, the annihilator of $M$ is the ideal of $R$ defined by
\begin{equation*}
\ann_R(M):=\{ r\in R \mid rm=0 \,\,\text{for all } m\in M \}. 
\end{equation*}
An ideal $\gp$ of a ring $R$ is \emph{primitive} if it is the annihilator of some simple $R$-module. The set of primitive ideals of $R$ is called the primitive spectrum of $R$, and is denoted by $\Prim(R)$.  It is well known that every primitive ideal is prime and every maximal ideal is primitive. For a subset $\CS \subset R$, we write $\langle \CS\rangle$ for the two-sided ideal generated by $\CS$. 

\begin{theorem} \label{a26May26}  
\begin{enumerate}
	\item The prime spectrum of $J$ is 
	\begin{equation*}
	\Spec(J)=\{ \langle 0 \rangle \}\,\,\cup \,\,\{ \langle \Omega-\alpha \rangle \mid \alpha \in \mK^* \}\,\,\cup \,\,\{ \langle x, \gp \rangle \mid \gp \in \Spec(\mK[g^{\pm 1},y]) \}.
	\end{equation*}
	The inclusions among these prime ideals are depicted in the following diagram:
	\[
	\begin{tikzpicture}[scale=0.5]
	
	\node (zero) at (0,0) {$\langle 0\rangle$};
	
	\node (x) at (-4,2)
	{$\langle x\rangle$};
	
	\node (omega) at (4,2)
	{\qquad $\{\langle \Omega-\alpha \rangle \mid \alpha \in\mK^*\}$};
	
	\node (xp) at (-4,4)
	{$\{\langle x, \gp\rangle \mid \gp\in\operatorname{Spec}(\mK[g^{\pm1},y]),\,
		\operatorname{ht}(\gp)=1\}$};
	
	\node (max) at (-4,6)
	{$\{\langle x,g-\beta,y-\gamma\rangle
		\mid \beta\in\mK^*,\ \gamma\in\mK\}$};
	
	\draw (zero) -- (x);
	\draw (zero) -- (omega);
	
	\draw (x) -- (xp);
	
	\draw[dotted] (xp) -- (max);
	
	\end{tikzpicture}
	\]
	
	\item 	The primitive spectrum of $J$ coincides with its maximal spectrum and is given by
	\begin{equation*}
	\Prim(J)=\Max(J)=\{ \langle \Omega-\alpha \rangle \mid \alpha \in \mK^* \} \,\, \cup \,\, \{ \langle x, g-\beta, y-\gamma \rangle \mid \beta \in \mK^*, \gamma \in \mK \}. 
	\end{equation*}
	
	\item Every prime ideal of $J$ is completely prime. 
	
	\item Every nonzero ideal of $J$ intersects the centre $Z(J)$ nontrivially.
	
\end{enumerate}	
\end{theorem}
\begin{proof}
	(1) We first partition $\Spec(J)$ according to whether a prime ideal contains $x$. Let $\Spec(J,x)$ (resp. $\Spec_x(J)$) denote the set of prime ideals containing (resp. not containing) $x$. Then 
	\begin{equation*}
	\Spec(J)=\Spec(J, x)\,\cup \, \Spec_x(J).
	\end{equation*}
	
	Prime ideals containing $x$ correspond bijectively to prime ideals of $J/\langle x \rangle \cong \mK[g^{\pm 1},y]$. Thus $\Spec(J,x)=\{ \langle x, \gp \rangle \mid \gp \in \Spec(\mK[g^{\pm 1},y]) \}$. Since $x$ is normal in $J$, prime ideals not containing $x$ correspond bijectively to prime ideals of the localization $J_x$. By \eqref{J0x}, $J_x \cong \mK[\Omega^{\pm 1}] \otimes \mK[x^{\pm 1}][y';\d]$, where $\mK[x^{\pm 1}][y';\d]$ is simple with centre $\mK$. Hence the prime ideals of $J_x$ correspond bijectively to those of $\mK[\Omega^{\pm1}]$, namely
	\begin{equation*}
	\Spec(J_x)=\{ \langle 0 \rangle \} \,\,\cup \,\, \{ (\Omega-\alpha)J_x \mid \alpha \in \mK^* \}. 
	\end{equation*}
	The zero ideal of $J_x$ contracts to $\langle 0 \rangle$ in $J$. For $\alpha\in \mK^*$, consider the central quotient
	\begin{equation*}
	J^{\alpha}:=J/(\Omega-\alpha)J. 
	\end{equation*}
	Since $\alpha\neq 0$, the relation $x^2=\Omega g$ implies that $x$ becomes invertible in $J^\alpha$. Hence localization does not change the algebra:
	\begin{equation*}
	J^{\alpha} \cong J^{\alpha}_x \cong J_x/(\Omega-\alpha)J_x.
	\end{equation*}
	Using the structure of $J_x$ (see \eqref{J0x}), we obtain
	\begin{equation} \label{Jalph} 
	J^{\alpha } \cong \mK[x^{\pm 1}][y';\d], \quad \d(x)=-\tfrac{1}{2}x,
	\end{equation}
	which is a central simple algebra. Hence $J^\alpha$ is simple, so $(\Omega-\alpha)J$ is a maximal ideal of $J$. Moreover, contraction gives $J \cap (\Omega-\alpha)J_x=(\Omega-\alpha)J$. Therefore, $\Spec_x(J)=\{ \langle 0 \rangle \} \,\cup \, \{ \langle \Omega-\alpha \rangle \mid \alpha \in \mK^* \}$. Combining both cases gives the full description of $\Spec(J)$. The inclusions of prime ideals are obvious.
	
	(2) The description of maximal ideals follows from the inclusions of primes given in statement (1). All ideals in $\Spec(J,x)$ correspond to primes in the commutative algebra $\mK[g^{\pm1},y]$, hence are primitive exactly when maximal. Since all maximal ideals are primitive, and the zero ideal $\langle 0 \rangle$ is not primitive because $Z(A)=\mK[\Omega]$, we have $\Prim(J)=\Max(J)$.
	
	(3) This follows from statement (1). 
	
	(4) Let $\ga \neq 0$ be an ideal of $J$. Since $J$ is Noetherian, there exists prime ideals $\gp_1,\ldots, \gp_n$ in $J$ such that $\gp_1 \cdots \gp_n \subseteq \ga$. By statement (1), every nonzero prime of $J$ intersects $Z(J)$ nontrivially,  and therefore so does $\ga$. 
\end{proof}

\subsection{Simple $J$-modules}
For an algebra $R$, let $\widehat{R}$ denote the set of isomorphism classes $[M]$ of simple $R$-modules $M$.  If $\mathcal{P}$ is a property of modules that is invariant under module isomorphisms, then we write $\widehat{R}(\mathcal P)$ for the subset of $\widehat{R}$ consisting of all isomorphism classes of simple $R$-modules satisfying $\mathcal{P}$.
\begin{corollary} \label{SimMod}
	Assume that $\mK$ is algebraically closed of characteristic zero. 
	\begin{enumerate}
		\item 	 A complete set of pairwise non-isomorphic finite-dimensional simple $J$-modules is
		\begin{equation*}
		\{[S_{\beta,\gamma}] \mid \beta \in \mK^*, \gamma \in \mK\},\quad \text{where } S_{\beta,\gamma}:=J/\langle x, g-\beta, y-\gamma \rangle. 
		\end{equation*}
		In particular, every finite-dimensional simple $J$-module is one-dimensional.
		\item 	A simple $J$-module $M$ is infinite-dimensional if and only if the central element $\Omega$ acts on $M$ by a nonzero scalar. Thus
		\begin{equation*}
		\widehat{J}({\text{\rm infinite-dimensional}})=\bigsqcup_{\alpha \in \mK^*} \widehat{J^{\alpha}}.  
		\end{equation*}
	\end{enumerate}
\end{corollary}

\begin{proof}
	(1) Let $M$ be a finite-dimensional simple $J$-module. Then $\ann_J(M)$ is a primitive ideal of finite-codimension. By Theorem \ref{a26May26}(2), we have $\ann_J(M)=\langle x, g-\beta, y-\gamma \rangle$ for some $\beta \in \mK^*$ and $\gamma \in \mK$. It follows that $M \cong S_{\beta,\gamma}$. 
	
	(2) Let $M$ be an infinite-dimensional simple $J$-module. By Theorem \ref{a26May26}(2), $\ann_J(M)=\langle \Omega-\alpha \rangle$ for some $\alpha \in\mK^*$. Thus $M$ is a simple module over $J^{\alpha}=J/(\Omega-\alpha)J$.	
	By \eqref{Jalph}, $J^{\alpha}\cong \mK[x^{\pm 1}][y';\d]$ where $\d(x)=-\tfrac{1}{2}x$, which is a simple infinite-dimensional algebra. Hence every simple $J^{\alpha}$-module is infinite-dimensional. 
	
	Finally, let $M$ be any simple $J$-module. Since $\Omega$ is central and $\mK$ is algebraically closed, Schur's lemma implies that $\Omega$ acts on $M$ as a scalar $\alpha\in\mK$. If $\alpha=0$, then Theorem~\ref{a26May26}(2) shows that $M$ is finite-dimensional. Thus $M$ is infinite-dimensional if and only if $\alpha\in\mK^*$, and the stated decomposition follows.
\end{proof}

\begin{remark}
	Corollary \ref{SimMod}(2) reduces the classification of infinite-dimensional simple $J$-modules to the classification of simple modules over the algebras $J^\alpha$ with $\alpha\in\mK^*$. By \eqref{Jalph}, each $J^\alpha$ is isomorphic to the Ore extension
	$$
	\mK[x^{\pm1}][y';\delta],
	\qquad
	\delta(x)=-\tfrac12 x.
	$$
	Setting $t=-2y'x^{-1}$, one obtains $[t,x]=1$, and hence $J^{\alpha}\cong A_{1}(\mK)_x$, 
	the localization of the first Weyl algebra $A_1(\mK)=\mK\langle x,t\mid tx-xt=1\rangle$ at the powers of $x$. Therefore, the classification of infinite-dimensional simple $J$-modules is equivalent to the classification of simple modules over a localization of the first Weyl algebra. A classification of simple modules over the Weyl algebra was obtained by Block \cite{Block}, with further developments in \cite{Bavula-OreD}.
\end{remark}

\subsection{The automorphism group of $J$}
We determine the automorphism group of the algebra $J$.
There are obvious subgroups of $\Aut(J)$:
\begin{equation*}
\begin{aligned}
\mathbb{T}^2&=\{t_{\alpha, \beta}\mid \alpha,\beta\in \mK^* \}, \quad &t_{\alpha,\beta}:& \quad g \mapsto \alpha g, \quad x \mapsto \beta x, \quad y \mapsto \beta y. \\
\mathbb{I}&=\{1, \iota \}, \quad &\iota: &\quad g \mapsto g^{-1}, \quad x \mapsto g^{-1}x, \quad y \mapsto -g^{-1}y, \\
\mathbb{A}&=\{ \tau_f \mid f\in \mK[g^{\pm 1},x] \}, \quad &\tau_f:& \quad g \mapsto g, \quad x \mapsto x, \quad y \mapsto y+f. 
\end{aligned}
\end{equation*}
The group $\mathbb{T}^2$ is a two-dimensional algebraic torus, since $\mathbb{T}^2 \cong (\mK^*)^2$. The group $\mathbb{I}$ has order $2$, as $\iota^2=\mathrm{id}$.  Moreover, $\mathbb{A}$ is isomorphic to the
additive group of the algebra $\mK[g^{\pm 1},x]$.  

We will use the following well-known result. 
\begin{lemma} \label{a11Jun26}   
	Let $r$ and $s$ be nonzero normal elements of a domain $R$ such that $\langle r \rangle=\langle s \rangle$. Then there exists a unit $u\in R$ such that $r=us$. 
\end{lemma}
\begin{proof}
	Since $r$ and $s$ are normal, there exist $u,v\in R$ such that $r=us$ and $s=vr.$
	Hence $r=uvr$, so $(1-uv)r=0$. As $R$ is a domain, it follows that $uv=1$. Similarly, $s=vus$, whence $(1-vu)s=0$, and therefore $vu=1$.  Thus $u$ is a unit with inverse $v$, and $r=us$.
\end{proof}
 
\begin{theorem}
	The automorphism group of $J$ is $\Aut(J)=\mathbb{T}^2 \ltimes \mathbb{I} \ltimes \mathbb{A}$. 
\end{theorem}

\begin{proof}
	Clearly, the group $\Aut(J)$ contains the semidirect product $\mathbb{T}^2 \ltimes \mathbb{I} \ltimes \mathbb{A}$. We prove the reverse inclusion.
	
	Let $\s \in \Aut(J)$. Since the group of units of $J$ is $U(J)=\{\mK^*g^i \mid i\in \Z \}$, and automorphisms preserve units, we have $\s(g)=\alpha g$ or $\s(g)=\alpha g^{-1}$ for some $\alpha \in \mK^*$. Composing with $\iota$ if necessary, we may assume that $\s(g)=\alpha g$. 
	
	By Theorem \ref{a26May26}, the ideal $\langle x\rangle$ is the unique non-maximal height-one prime ideal of $J$. Hence $\s(\langle x \rangle)=\langle x \rangle$. Since $x$ is normal, Lemma~\ref{a11Jun26} implies that $\s(x)=u x$ for some unit $u\in J$. Therefore $\s(x)=\beta g^ix$ for some $\beta \in \mK^*$ and $i\in \Z$. 	
	Since automorphisms preserve the centre, there exist
	$\mu\in\mK^*$ and $\nu\in\mK$ such that
	 $\s(\Omega)=\mu \Omega+\nu$. Using $\Omega=x^2 g^{-1}$, we obtain
	 $\alpha^{-1}\beta^2 g^{2i-1}x^2=\mu g^{-1}x^2+\nu$. Comparing both sides yields $i=0$, $\nu=0$ and $\mu=\alpha^{-1}\beta^2$. Hence $\s(x)=\beta x$. 
	 
	 Applying $\s$ to the relation $gyg^{-1}=y+x$ gives $g\s(y)g^{-1}=\s(y)+\beta x$. It follows that
	 \begin{equation*}
	g\big(\s(y)-\beta y\big)g^{-1}=\s(y)-\beta y.
	 \end{equation*} 
	 Hence $\s(y)-\beta y$ belongs to the centralizer $C_J(g)$. By Lemma \ref{a10Jun26}, $C_J(g)=\mK[g^{\pm1},x]$, and therefore, $\s(y)=\beta y+f$ for some $f\in \mK[g^{\pm 1},x]$. It follows that $\s= \tau_{\beta^{-1}f} t_{\alpha,\beta}$. Thus $\Aut(J) \subseteq \mathbb{T}^2 \ltimes \mathbb{I} \ltimes \mathbb{A}$. This completes the proof. 		 
\end{proof}

\section{Prime ideals, representations, and automorphisms of  $\mathfrak{U}(\l)$}  \label{AlgU}  

In this section, we study the algebra $\mathfrak{U}(\l)$ under the assumption that $\lambda \neq 0$. We determine its centre, describe its prime and primitive spectra, and compute its automorphism group. We also construct explicit simple modules whose annihilators realize all primitive ideals. Recall that $\mathfrak{U}(\l) \cong \mathfrak{U}(\l')$ for all $\l, \l' \in \mK^*$.  Throughout, we write $\mathfrak{U}:=\mathfrak{U}(\l)$. 

\subsection{The centre of $\mathfrak{U}(\l)$}
We introduce the following elements
\begin{equation} \label{phipsi}  
\begin{aligned}
\phi_{+}:&=x+\sqrt{2\l}(g-1),\quad  & \phi_{-}:&=x-\sqrt{2\l}(g-1),\\
\psi_{+}:&=x+\sqrt{2\l}(g+1), &\psi_{-}:&=x-\sqrt{2\l}(g+1). 
\end{aligned}
\end{equation}
Then one has
\begin{equation*}
\psi_{+}=\phi_{+}+2\sqrt{2\l}, \qquad \psi_{-}=\phi_{-}-2\sqrt{2\l}. 
\end{equation*}
A straightforward computation yields
\begin{equation*}
\begin{aligned}\
[y, \phi_{+}]=-\frac{1}{2} \phi_{+} \psi_{+}=-\frac{1}{2}\phi_{+}(\phi_{+}+2\sqrt{2\l}), \\
[y, \phi_{-}]=-\frac{1}{2} \phi_{-} \psi_{-}=-\frac{1}{2}\phi_{-}(\phi_{-}-2\sqrt{2\l}),\\
[y, \psi_{+}]=-\frac{1}{2} \psi_{+} \phi_{+}=-\frac{1}{2}\psi_{+}(\psi_{+}-2\sqrt{2\l}), \\
[y, \psi_{-}]=-\frac{1}{2} \psi_{-} \phi_{-}=-\frac{1}{2}\psi_{-}(\psi_{-}+2\sqrt{2\l}).
\end{aligned}
\end{equation*}
It follows that the elements $\phi_{\pm}$ and $\psi_{\pm}$ are normal in $\mathfrak{U}$. Furthermore, 
\begin{equation} \label{Opp} 
\Omega+4\l= \phi_{+} \phi_{-}g^{-1}, \qquad \Omega-4\l=\psi_{+}\psi_{-}g^{-1}. 
\end{equation}

Let $\mathfrak{U}_{\CS}$ denote the localization of $\mathfrak{U}$ at the Ore set $\CS:=\mK[g^{\pm 1}, x] \setminus\{0\}$, and set 
\begin{equation*}
K:=\Frac(\mK[g^{\pm 1},x])=\mK(g,x).
\end{equation*}
 By \eqref{OreJ}, the localized algebra $\mathfrak{U}_{\CS}$ is a skew polynomial algebra over $K$, namely 
\begin{equation*}
\mathfrak{U}_{\CS}=K[y;\d]
\end{equation*}
where $\d$ is the derivation of $K$ determined by $\d(g)=-xg$ and $\d(x)=-\tfrac{1}{2}x^2+\l(1-g^2)$. 

We claim that $K=\mK(\Omega, \phi_{+})$. The inclusion $\mK(\Omega, \phi_{+}) \subseteq K$ is immediate. To establish the reverse inclusion, it suffices to show that $g$ and $x$ belong to $\mK(\Omega, \phi_{+})$. Note that $\psi_{+}=\phi_{+}+2\sqrt{2\l}\in \mK(\Omega, \phi_{+})$.  Using \eqref{Opp}, we obtain
\begin{equation*}
(\Omega+4\l)\phi_{+}^{-1}-(\Omega-4\l)\psi_{+}^{-1}=(\phi_{-}-\psi_{-})g^{-1}=2\sqrt{2\l} g^{-1} \in \mK(\Omega, \phi_{+}). 
\end{equation*}
Thus
\begin{equation*}
g=\frac{2\sqrt{2\l} \phi_{+}\psi_{+}}{(\Omega+4\l)\psi_{+}-(\Omega-4\l)\phi_{+}} \in \mK(\Omega, \phi_{+}). 
\end{equation*}
Next, using the first identity in \eqref{Opp}, we have
\begin{equation*}
\phi_{-}=(\Omega+4\l)\phi_{+}^{-1}g \in \mK(\Omega, \phi_{+}). 
\end{equation*}
Hence
\begin{equation*}
x=\frac{1}{2}(\phi_{+}+\phi_{-}) \in \mK(\Omega, \phi_{+}). 
\end{equation*}
This proves $K=\mK(\Omega, \phi_{+})$. Consequently,
\begin{equation} \label{JSOre} 
\mathfrak{U}_{\CS}=K[y;\d]=\mK(\Omega)(\phi_{+})[y;\d]
\end{equation}
where $\d$ is the derivation of $K$ determined by $\d(\Omega)=0$ and $\d(\phi_{+})=-\frac{1}{2}\phi_{+}(\phi_{+}+2\sqrt{2\l})$.

The following lemma determines the center of $\mathfrak{U}$. 
\begin{lemma}
Let $\mathfrak{U}=\mathfrak{U}(\l)$ with $\l \neq 0$. Then $Z(\mathfrak{U}_{\CS})=\mK(\Omega)$ and $Z(\mathfrak{U})=\mK[\Omega]$. 
\end{lemma}
\begin{proof}
	Define $y':=-2y \phi_{+}^{-1}(\phi_{+}+2\sqrt{2\l})^{-1} \in \mathfrak{U}_{\CS}$. Then $[y', \phi_{+}]=1$. In view of \eqref{JSOre}, it follows that
	\begin{equation*}
	\mathfrak{U}_{\CS}=K[y';\d']
	\end{equation*}
	where $\d'$ is the derivation of $K$ determined by $\d'(\Omega)=0$ and $\d'(\phi_{+})=1$. 
	
	Let $z=\sum_{i=0}^n a_i y'^i \in Z(\mathfrak{U}_{\CS})$ where $a_i \in K$. Since $z$ commutes with $\phi_+$, we have
	\begin{equation*}
	0=[z, \phi_{+}]=\sum_{i=0}^n ia_iy'^{i-1}.    
	\end{equation*}
	This forces $a_i=0$ for all $i\geq 1$, hence $z \in K$. Since $z$ also commutes with $y'$, we obtain $\delta'(z)=0$. Hence $Z(\mathfrak{U}_{\mathcal S}) = K^{\delta'} = \mK(\Omega)$. 
	
	Now let $z \in Z(\mathfrak{U})$. Then $z \in \mathfrak{U} \cap Z(\mathfrak{U}_{\mathcal S}) = \mathfrak{U} \cap \mK(\Omega)$.  Write $z=p(\Omega)q(\Omega)^{-1}$ where $p,q\in \mK[\Omega]$ are coprime. Suppose that $q$ is nonconstant, and choose $\mu\in\mK$ such that $q(\mu)=0$. Since $(p,q)=1$, we have $p(\mu)\neq 0$.
	Reducing the identity $p(\Omega)=zq(\Omega)$ modulo $(\Omega-\mu)\mathfrak{U}$ yields $p(\mu)=0$, a contradiction. Hence $q(\Omega)\in \mK^*$, and therefore $z\in \mK[\Omega]$. Thus $Z(\mathfrak{U})\subseteq \mK[\Omega]$. Since $\Omega$ is central, the reverse inclusion is immediate. Consequently, $Z(\mathfrak{U})=\mK[\Omega]$.
\end{proof}

The next lemma shows that every nonzero prime ideal of $\mathfrak{U}$ intersects the centre  nontrivially.
\begin{lemma} \label{1Jun26} 
	 If $P$ is a nonzero prime ideal of $\mathfrak{U}$, then $P \cap \mK[\Omega] \neq 0$. 
\end{lemma}
\begin{proof}
	Recall from \eqref{OreJ} that $\mathfrak{U}=\mK[g^{\pm 1},x][y;\d]$. By \cite[Lemmas 3.19 and 3.21]{Goodearl--Warfield}, the ideal $\gp:=P \cap \mK[g^{\pm 1},x]$ is a nonzero $\d$-stable prime ideal of $\mK[g^{\pm 1}, x]$. We distinguish two cases.
	
	\textbf{Case 1: $x \in \mathfrak{p}$.}  Since $\d(x)=-\tfrac{1}{2}x^2+\l(1-g^2) \in \gp$, we obtain $1-g^2\in \gp$. As $\gp$ is prime, it follows that either $g-1\in \gp$ or $g+1\in \gp$. If $g-1\in \gp$, then $\phi_{+}, \phi_{-} \in \gp$, and hence $\Omega+4\l=\phi_{+}\phi_{-}g^{-1}\in \gp$. If $g+1\in \gp$, then $\psi_{+}, \psi_{-} \in \gp$, and therefore $\Omega-4\l=\psi_{+}\psi_{-}g^{-1}\in \gp$. In either case, $P \cap \mK[\Omega] \neq 0$.

	\textbf{Case 2: $x \notin \mathfrak{p}$.}  Since $\mK[g^{\pm 1}, x]=\mK[g^{\pm 1}, \Omega]\oplus \mK[g^{\pm 1}, \Omega]x$, the ring $\mK[g^{\pm 1}, x]$ is integral over $\mK[g^{\pm 1}, \Omega]$. As $\gp \neq 0$, the incomparability theorem (\cite[Corollary 4.18]{Eisenbud}) implies $\gp \cap \mK[g^{\pm 1}, \Omega] \neq 0$. 
	Choose a nonzero element $u=\sum_{i=0}^n r_i g^i \in \gp \cap \mK[g^{\pm 1}, \Omega]$ with $n$ minimal, where $r_i \in \mK[\Omega]$ and $r_n \neq 0$.  Note that for every $i \geq 1$, $[y, g^i]=-ix g^i.$
	Since $\gp$ is $\d$-stable, 
	\begin{equation*}
	[y, u]=-x\sum_{i=0}^n ir_i g^i \in \gp.  
	\end{equation*}
	Because $x\notin \gp$ and $\gp$ is prime, it follows that $u':=\sum_{i=0}^n ir_i g^i \in \gp$. Hence,
	\begin{equation*}
	w:=nu-u'=\sum_{i=0}^{n-1}(n-i)r_i g^i \in \gp. 
	\end{equation*}
	Since $w \in \mK[g^{\pm1},\Omega]$ and $\deg_g(w)<n$, the minimality of $n$ forces $n=0$. Thus $u=r_0 \in \gp \cap \mK[\Omega]$. Therefore, $P \cap \mK[\Omega] \neq 0$, as required. 
\end{proof}

\subsection{Some factor algebras}
For any $\alpha \in \mK$, consider the central quotient
\begin{equation*}
\mathfrak{U}^{\alpha}:=\mathfrak{U}/\langle \Omega-\alpha \rangle. 
\end{equation*}
The relation $\Omega=\alpha$ yields
$x^2=2\lambda g^2+\alpha g+2\lambda$. Then in $\mathfrak{U}^{\alpha}$ we have 
\begin{equation*}
[y,x]
=-\frac12x^2+\lambda(1-g^2)
=-2\lambda g^2-\frac12\alpha g.
\end{equation*}

Define the following elements in $\mathfrak{U}^{\alpha}$:
\begin{equation} \label{hef} 
h:=\frac{1}{\sqrt{2\l}}yg^{-1}, \qquad e:=x-\sqrt{2\l}\,g-\frac{\alpha}{2\sqrt{2\l}}, \qquad f:=x+\sqrt{2\l}\,g+\frac{\alpha}{2\sqrt{2\l}}. 
\end{equation}
A straightforward computation shows that
\begin{equation*}
\begin{gathered}\
[h, e]=e, \qquad [h, f]=-f, \qquad [e, f]=0, \\
ef=\frac{1}{8\l}(16\l^2-\alpha^2). 
\end{gathered}
\end{equation*}
Hence,  if $\alpha \neq \pm 4\l$, then $e$ and $f$ are invertible in $\mathfrak{U}^{\alpha}$.

Let $\CB$ denote the subalgebra of $\mathfrak{U}^{\alpha}$ generated by $h, e$, and $f$. Then $\CB$ admits the presentation
\begin{equation} \label{Bef}  
\CB \cong \frac{\mK[e,f]}{\langle ef-\frac{1}{8\l}(16\l^2-\alpha^2)\rangle}\Big[h;\partial\Big]
\end{equation}
where $\partial$ is the derivation determined by $\partial(e)=e$ and $\partial(f)=-f$.
In particular, $\CB$ is a domain if and only if $\alpha \neq \pm 4\l$.  
The following proposition shows that $\mathfrak{U}^{\alpha}$ is a localization of $\CB$.

\begin{proposition} \label{5Jun26}  
	Let $\mathfrak{U}=\mathfrak{U}(\l)$ with $\l \neq 0$. 
	\begin{enumerate}
		\item Suppose that $\alpha \neq \pm 4\l$. Then $\CB\cong \mK[e^{\pm 1}][h;\partial]$ where $\partial$ is the derivation of $\mK[e^{\pm 1}]$ determined by $\partial(e)=e$. Furthermore, the algebra $\mathfrak{U}^{\alpha}$ is isomorphic to the localization of $\CB$ at the powers of the element $\tilde{g}:=e^2+\frac{\alpha}{\sqrt{2\l}}e+\frac{1}{8\l}(\alpha^2-16\l^2)$. More precisely, 
		\begin{equation*}
		\mathfrak{U}^{\alpha} \cong \CB_{\tilde{g}} \cong \mK[e^{\pm 1}]_{\tilde{g}}[h;\partial].
		\end{equation*}
		In particular, $\mathfrak{U}^{\alpha}$ is a simple domain. 
		
		\item Suppose that $\alpha=\pm 4\l$. Then $\CB \cong \frac{\mK[e,f]}{\langle ef \rangle}\Big[h;\partial \Big]$ where $\partial$ is the derivation determined by $\partial(e)=e$ and $\partial(f)=-f$. Furthermore,  the algebra $\mathfrak{U}^{\alpha}$ is isomorphic to the localization of $\CB$ at the powers of the element $g':=e-f+\frac{\alpha}{\sqrt{2\l}}$. More precisely, 
		\begin{equation*}
		\mathfrak{U}^{\alpha} \cong \CB_{g'} \cong \Big(\frac{\mK[e,f]}{\langle ef \rangle}\Big)_{g'}\Big[h;\partial \Big]. 
		\end{equation*}
		In this case, $\mathfrak{U}^{\alpha}$ is neither simple nor a domain. 
	\end{enumerate}
\end{proposition}
\begin{proof}
	(1) Assume that $\alpha \neq \pm 4\l$. Then the elements $e$ and $f$ are invertible in $\CB$. It follows from \eqref{Bef} that $\CB \cong \mK[e^{\pm 1}][h;\partial]$, which is a simple domain. 
	By \eqref{hef}, 
	\begin{equation*}
	x=\frac{1}{2}(e+f), \qquad g=-\frac{1}{2\sqrt{2\l}}(e-f+\frac{\alpha}{\sqrt{2\l}}). 
	\end{equation*}
	Hence $x,g\in \CB$, and therefore $y=\sqrt{2\l}hg\in \CB$. Moreover, using $ef=\frac{1}{8\l}(\alpha^2-16\l^2)$, we obtain
	\begin{equation*}
	g=-\frac{1}{2\sqrt{2\l}}e^{-1}(e^2-ef+\frac{\alpha}{\sqrt{2\l}}e)= -\frac{1}{2\sqrt{2\l}}e^{-1} \tilde{g}. 
	\end{equation*}
	Since $e$ is invertible, it follows that $g^{-1}\in \CB_{\tilde{g}}$. Consequently,  $x, g^{\pm 1}, y\in \CB_{\tilde{g}}$, and hence $\mathfrak{U}^{\alpha} \subseteq \CB_{\tilde{g}}$. The reverse inclusion is immediate, so $\mathfrak{U}^{\alpha} = \CB_{\tilde{g}}$. Since $\CB$ is a simple domain,  its localization $\mathfrak{U}^{\alpha}$ is also a simple domain. 
	
	(2) Assume that $\alpha=\pm 4\lambda$. Then $\alpha^2=16\lambda^2$, and \eqref{Bef} yields $\CB \cong \frac{\mK[e,f]}{\langle ef \rangle}\Big[h;\partial \Big]$. As in the previous case, we have $x, g, y \in \CB$. Since $g=-\frac{1}{2\sqrt{2\l}} g'$, it follows that $g^{-1}\in \CB_{g'}$. Thus the generators $x, g^{\pm 1}, y$ belong to $\CB_{g'}$, whence $\mathfrak{U}^{\alpha} \subseteq \CB_{g'}$. The reverse inclusion is immediate, and thus $\mathfrak{U}^{\alpha}=\CB_{g'}$. Since $ef=0$, the algebra $\mathfrak{U}^{\alpha}$ is not a domain. Furthermore, the ideals generated by the normal elements $e$ and $f$ are proper nonzero ideals, so $\mathfrak{U}^{\alpha}$ is not simple.
\end{proof}

Recall that the elements $\phi_{\pm}=x \pm \sqrt{2\l}(g-1)$ and $\psi_{\pm}=x\pm \sqrt{2\l}(g+1)$ are normal in $\mathfrak{U}$. Hence the corresponding principal left/right ideals are two-sided, and the corresponding factor algebras are obtained by imposing the relations
\begin{equation*}
\begin{aligned}
 \phi_{+}=0 \quad  &\Rightarrow \quad  x=-\sqrt{2\l}(g-1),\\
 \phi_{-}=0 \quad  &\Rightarrow \quad  x=\sqrt{2\l}(g-1), \\
 \psi_{+}=0 \quad  &\Rightarrow \quad  x=-\sqrt{2\l}(g+1), \\
 \psi_{-}=0 \quad  &\Rightarrow \quad  x=\sqrt{2\l}(g+1).
\end{aligned}
\end{equation*}

The next proposition identifies the factor algebras corresponding to the normal elements $\phi_{\pm}$ and $\psi_{\pm}$ as Ore extensions over the Laurent polynomial algebra $\mK[g^{\pm1}]$.
\begin{proposition} \label{2Jun26}   
 	Each of the algebras $\mathfrak{U}/\phi_{\pm}\mathfrak{U}$ and $\mathfrak{U}/\psi_{\pm}\mathfrak{U}$ is an Ore extension over the Laurent polynomial algebra $\mK[g^{\pm 1}]$. More precisely, 
	\begin{enumerate}
		\item For the $\phi$-quotients, 
		\begin{equation*}
		\begin{aligned}
			\mathfrak{U}/\phi_{+}\mathfrak{U} &\cong \mK[g^{\pm 1}][y;\d_{+}], \qquad \text{where } \,\,\d_{+}(g)= \sqrt{2\l}\,g(g-1),  \\
		\mathfrak{U}/\phi_{-}\mathfrak{U} &\cong \mK[g^{\pm 1}][y;\d_{-}], \qquad  \text{where }\,\,\d_{-}(g)=-\sqrt{2\l} \,g(g-1). 
		\end{aligned}
		\end{equation*}
		Furthermore, $\mathfrak{U}/\phi_{+}\mathfrak{U}$ and $\mathfrak{U}/\phi_{-}\mathfrak{U}$ are isomorphic via $y \mapsto -y$. 
		\item For the $\psi$-quotients, 
		\begin{equation*}
		\begin{aligned}
		\mathfrak{U}/\psi_{+}\mathfrak{U} &\cong \mK[g^{\pm 1}][y;\d'_{+}], \qquad \text{where } \,\,\d'_{+}(g)= \sqrt{2\l}\,g(g+1),  \\
		\mathfrak{U}/\psi_{-}\mathfrak{U} &\cong \mK[g^{\pm 1}][y;\d'_{-}], \qquad  \text{where }\,\,\d'_{-}(g)=-\sqrt{2\l} \,g(g+1). 
		\end{aligned}
		\end{equation*}
		Furthermore, $\mathfrak{U}/\psi_{+}\mathfrak{U}$ and $\mathfrak{U}/\psi_{-}\mathfrak{U}$ are isomorphic via $y \mapsto -y$. 
	\end{enumerate}
\end{proposition}
\begin{proof}
	Consider first the quotient $\mathfrak{U}/\phi_{+}\mathfrak{U}$. Substituting
	$x=-\sqrt{2\lambda}(g-1)$
	into the relation
	$[y,g]=-xg$
	gives
	$[y,g]=\sqrt{2\lambda}\,g(g-1).$
	Hence $\mathfrak{U}/\phi_{+}\mathfrak{U}$ is generated by $g^{\pm1}$ and $y$ subject to the single relation
	$[y,g]=\delta_{+}(g),$
	where $\delta_{+}(g)=\sqrt{2\lambda}\,g(g-1)$. It follows that
	$\mathfrak{U}/\phi_{+}\mathfrak{U}\cong\mK[g^{\pm1}][y;\delta_{+}].$	
	The remaining cases are analogous. Substituting the corresponding expressions for $x$ yields the stated derivations $\delta_{-}$, $\delta'_{+}$, and $\delta'_{-}$. In each case, the relation
	$[y,x]=-\frac12x^2+\lambda(1-g^2)$
	is automatically satisfied after substitution, so no further relations arise.
	
	Finally, replacing $y$ by $-y$ sends $\delta_{+}$ to $\delta_{-}$ and $\delta'_{+}$ to $\delta'_{-}$, yielding the asserted isomorphisms.
\end{proof}


\begin{remark}
	The elements $\phi_{\pm}$ and $\psi_{\pm}$ are $(g,1)$-primitive:
	\begin{equation*}
	\D(\phi_{\pm})=\phi_{\pm} \otimes 1+ g\otimes \phi_{\pm}, \qquad \D(\psi_{\pm})=\psi_{\pm} \otimes 1+ g\otimes \psi_{\pm}. 
	\end{equation*}
	It follows that the ideals $\phi_{\pm}\mathfrak{U}$ and $\psi_{\pm}\mathfrak{U}$ are Hopf ideals of $\mathfrak{U}$. Consequently, the quotient algebras $\mathfrak{U}/\phi_{\pm}\mathfrak{U}$ and $\mathfrak{U}/\psi_{\pm}\mathfrak{U}$ inherit Hopf algebra structures from $\mathfrak{U}$. In particular, the skew polynomial algebras described in Proposition~\ref{2Jun26} are Hopf algebras  in which $g$ is grouplike and $y$ is $(g,1)$-primitive.
\end{remark}

\begin{lemma} \label{3Jun26}  
	The prime spectra of the algebras $\mathfrak{U}/\phi_{\pm}\mathfrak{U}$ and $\mathfrak{U}/\psi_{\pm}\mathfrak{U}$ are given by 
	\begin{equation*}
	\begin{aligned}
		\Spec(\mathfrak{U}/\phi_{\pm}\mathfrak{U})=\{ \langle 0 \rangle, \,\langle g-1 \rangle \} \,\, \cup \,\, \{ \langle g-1, y-\gamma \rangle \mid \gamma \in \mK \},  \\
		\Spec(\mathfrak{U}/\psi_{\pm}\mathfrak{U})=\{ \langle 0 \rangle,\, \langle g+1 \rangle \} \,\, \cup \,\, \{ \langle g+1, y-\gamma \rangle \mid \gamma \in \mK \}.  
	\end{aligned}
	\end{equation*}
\end{lemma}
\begin{proof}
	Let $\bar{\mathfrak{U}}:=\mathfrak{U}/\phi_{{+}}\mathfrak{U}$. By Proposition \ref{2Jun26}(1), we have $\bar{\mathfrak{U}}=\mK[g^{\pm 1}][y;\d_{+}]$ where $\d_{+}(g)= \sqrt{2\l}\,g(g-1)$. In particular, $\bar{\mathfrak{U}}$ is a domain, so $\langle 0 \rangle$ is a prime ideal. Note that the element $g-1$ is normal in $\bar{\mathfrak{U}}$. Let $\bar{\mathfrak{U}}_{g-1}$ denote the localization of $\bar{\mathfrak{U}}$ at the powers of $g-1$, and set $y':=\frac{1}{\sqrt{2\l}}yg^{-1}(g-1)^{-1}$. A direct computation gives $[y',g]=1$. Hence $\bar{\mathfrak{U}}_{g-1} \cong \mK[g^{\pm 1}][y'; \d']$ where $\d'(g)=1$. Therefore $\bar{\mathfrak{U}}_{g-1}$ is a localization of the first Weyl algebra and is consequently simple. As $g-1$ is normal, it follows that every nonzero prime ideal of $\bar{\mathfrak{U}}$ contains $g-1$. Since $\bar{\mathfrak{U}}/(g-1)\bar{\mathfrak{U}}\cong \mK[y]$ and $\mK$ is algebraically closed, the prime ideals of $\bar{\mathfrak{U}}$ containing $g-1$ are precisely  $\{\langle g-1 \rangle \} \,\, \cup \,\, \{ \langle g-1, y-\gamma \rangle \mid \gamma \in \mK \}$. This gives the description of $\Spec(\bar{\mathfrak{U}})$. By Proposition \ref{2Jun26}(1), the algebras $\mathfrak{U}/\phi_{+}\mathfrak{U}$ and $\mathfrak{U}/\phi_{-}\mathfrak{U}$ are isomorphic, so the same description holds for $\mathfrak{U}/\phi_{-}\mathfrak{U}$. 
	
	The description of $\Spec(\mathfrak{U}/\psi_{\pm}\mathfrak{U})$ is obtained similarly, replacing $g-1$ by $g+1$. 
 \end{proof}

\subsection{The prime and primitive spectra}
The following theorem gives an explicit description of the prime, maximal, primitive, and completely prime ideals of $\mathfrak{U}(\l)$ with $\l \neq 0$. 
\begin{theorem} \label{A3Jun26} 
	Let $\mathfrak{U}=\mathfrak{U}(\l)$ with $\l \neq 0$. 
	\begin{enumerate}
		\item The prime spectrum of $\mathfrak{U}$ is 
		\begin{equation*}
		\begin{aligned}
				\Spec(\mathfrak{U})= \{ \langle 0 \rangle \} \,\,\cup &\,\, \{ \langle \phi_{+}\rangle,\, \langle \phi_{-} \rangle, \,\langle \psi_{+} \rangle,\, \langle \psi_{-} \rangle \} \,\,\cup \,\, \{ \langle \Omega-\alpha \rangle \mid \alpha \in \mK \setminus\{ \pm 4\l\}\}\\
				\cup& \,\,\{ \langle x, g-1 \rangle, \,\, \langle x, g+1 \rangle \}\\  \cup& \,\, \{ \langle x, g-1, y-\gamma \rangle \mid \gamma \in \mK  \} \,\, \cup \,\, \{ \langle x, g+1, y-\gamma \rangle \mid \gamma \in \mK  \}.
		\end{aligned}
		\end{equation*}
		The inclusions among these prime ideals are depicted in the following diagram:		
		
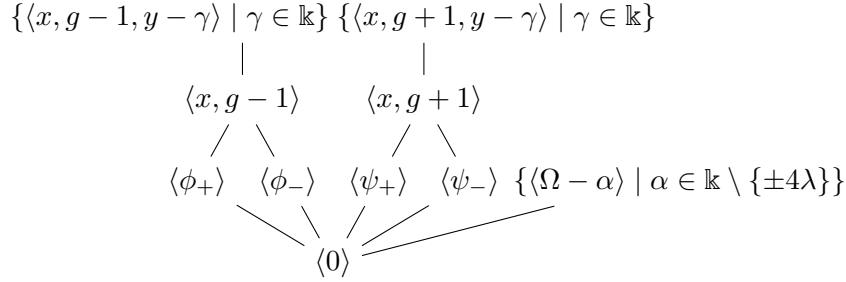
\begin{equation*}
		\begin{tikzpicture}[scale=0.6]
		
		\node (0) at (0,0) {$\langle0\rangle$};
		
		\node (p1) at (-3,1.8) {$\langle\phi_+\rangle$};
		\node (p2) at (-1,1.8) {$\langle\phi_-\rangle$};
		\node (p3) at ( 1,1.8) {$\langle\psi_+\rangle$};
		\node (p4) at ( 3,1.8) {$\langle\psi_-\rangle$};
		
		\node (q1) at (-2,3.6) {$\langle x,g-1\rangle$};
		\node (q2) at ( 2,3.6) {$\langle x,g+1\rangle$};
		
		\node (m1) at (-2,5.4)
		{$\{\langle x,g-1,y-\gamma\rangle\mid \gamma \in\mK\}\qquad\qquad\quad$};
		
		\node (m2) at ( 2,5.4)
		{$\qquad \qquad \quad \{\langle x,g+1,y-\gamma\rangle\mid \gamma\in\mK\}$};
		
		\node (o) at (7,1.8)
		{$\qquad \{\langle\Omega-\alpha\rangle
			\mid \alpha\in\mK\setminus\{\pm4\lambda\}\}$};
		
		\draw (0)--(p1);
		\draw (0)--(p2);
		\draw (0)--(p3);
		\draw (0)--(p4);
		\draw (0)--(o);
		
		\draw (p1)--(q1);
		\draw (p2)--(q1);
		
		\draw (p3)--(q2);
		\draw (p4)--(q2);
		
		\draw (q1)--(m1);
		\draw (q2)--(m2);
		
		\end{tikzpicture}
		\end{equation*}
		
		\item The maximal spectrum of $\mathfrak{U}$ is 
		\begin{equation*}
		\Max(\mathfrak{U})= \{ \langle \Omega-\alpha \rangle \mid \alpha \in \mK \setminus\{ \pm 4\l\}\}\,\cup \,\, \{\langle x, g\pm 1, y-\gamma \rangle \mid \gamma \in \mK  \}. 
		\end{equation*}
		
		\item The primitive spectrum of $\mathfrak{U}$ is
		\begin{equation*}
		\begin{aligned}
		\Prim(\mathfrak{U})= &\{ \langle \phi_{+}\rangle, \langle \phi_{-} \rangle, \langle \psi_{+} \rangle, \langle \psi_{-} \rangle \} \,\,\cup \,\, \{ \langle \Omega-\alpha \rangle \mid \alpha \in \mK \setminus\{ \pm 4\l\}\}\\
		&\cup \,\, \{ \langle x, g-1, y-\gamma \rangle \mid \gamma \in \mK  \} \,\, \cup \,\, \{ \langle x, g+1, y-\gamma \rangle \mid \gamma \in \mK  \}.
		\end{aligned}
		\end{equation*}
		
			\item Every prime ideal of $\mathfrak{U}$ is completely prime.
		
		\item Every nonzero ideal of $\mathfrak{U}$ intersects the centre $Z(\mathfrak{U})$ nontrivially.
	\end{enumerate}
\end{theorem}
\begin{proof}
	(1) Since $\mathfrak{U}$ is a domain, $\langle 0 \rangle$ is a completely prime ideal. Let $P$ be a nonzero prime ideal of $\mathfrak{U}$. By Lemma \ref{1Jun26}, we have $P \cap \mK[\Omega] \neq 0$. Since $\mK$ is algebraically closed,
	there exists $\alpha \in \mK$ such that $\Omega - \alpha \in P$.
	
	 If $\alpha \neq \pm 4 \l$, then by Proposition \ref{5Jun26}(1), the ideal $\langle \Omega-\alpha \rangle$ is maximal, hence $P=\langle \Omega-\alpha \rangle$. 
	 
	 If $\alpha=-4\l$ then $\Omega+4\l=\phi_{+}\phi_{-}g^{-1} \in P$. Since $\phi_{+}$ and $\phi_{-}$ are normal elements, primeness implies that $\phi_{+} \in P$ or $\phi_{-}\in P$. Prime ideals of $\mathfrak{U}$ containing $\phi_{\pm}$  correspond bijectively to prime ideals of $\mathfrak{U}/\phi_{\pm} \mathfrak{U}$. Moreover, $\langle \phi_{\pm}, g-1 \rangle=\langle x, g-1\rangle$.  Thus by Lemma \ref{3Jun26},  the prime ideals of $\mathfrak{U}$ containing $\Omega+4\l$ are 
	 \begin{equation*}
	 \{ \langle \phi_{+} \rangle, \langle \phi_{-} \rangle \}\,\,\cup \,\, \{ \langle x, g-1 \rangle \}\,\, \cup \,\, \{ \langle x, g-1, y-\gamma \rangle \mid \gamma \in \mK \}. 
	 \end{equation*}
	 
	 Similarly, if $\alpha=4\l$ then $\Omega-4\l=\psi_{+}\psi_{-}g^{-1}\in P$. Since $\psi_{+}$ and $\psi_{-}$ are normal, either $\psi_{+}\in P$ or $\psi_{-}\in P$. 
	 Prime ideals of $\mathfrak{U}$ containing $\psi_{\pm }$ correspond bijectively to prime ideals of $\mathfrak{U}/\psi_{\pm}\mathfrak{U}$.  Furthermore, $\langle \psi_{\pm}, g+1  \rangle=\langle x, g+1\rangle$.  Thus by Lemma \ref{3Jun26},   prime ideals of $\mathfrak{U}$ containing $\Omega-4\l$ are 
	 \begin{equation*}
	 \{ \langle \psi_{+} \rangle, \langle \psi_{-} \rangle \}\,\,\cup \,\, \{ \langle x, g+1 \rangle \}\,\, \cup \,\, \{ \langle x, g+1, y-\gamma \rangle \mid \gamma \in \mK \}. 
	 \end{equation*}
	 Combining all cases gives the description of $\Spec(\mathfrak{U})$.  All prime ideals are shown in the diagram, and the inclusions are obvious. 
	 
	 (2) The maximal ideals are precisely the maximal elements in the above poset, giving the stated description.
	 
	 (3) All maximal ideals are primitive. By \cite[Lemma 9.1.2 and Corollary 9.1.8]{MR}, $\mathfrak{U}$ is a Jacobson ring satisfying the Nullstellensatz over $\mK$. Hence every prime is an intersection of primitive ideals. From the inclusions of primes in statement (1), it follows that $\langle \phi_{\pm} \rangle$ and $\langle \psi_{\pm}  \rangle$ are primitive. The ideal $\langle 0 \rangle$ is not primitive since $\mathfrak{U}$ has nontrivial centre. The ideals $\langle x, g\pm 1\rangle$ are not primitive because $\mathfrak{U}/\langle x, g \pm 1 \rangle \cong \mK[y]$, which is commutative.
	 
	 (4) For all listed primes, the corresponding quotient algebras are domains,
	 by Proposition \ref{5Jun26}, Proposition \ref{2Jun26}, or direct inspection.
	 
	 (5) Let $\ga \neq 0$ be an ideal of $\mathfrak{U}$. Since $\mathfrak{U}$ is Noetherian, there exists prime ideals $P_1,\ldots, P_n$ such that $P_1 \cdots P_n \subseteq \ga$. By Lemma \ref{1Jun26}, every nonzero prime  ideal intersects $Z(\mathfrak{U})$ nontrivially,  hence so does $\ga$. 
\end{proof}

\subsection{Simple modules}
Having determined the primitive spectrum of $\mathfrak{U}(\l)$, we now construct explicit simple modules realizing these primitive ideals. 
More precisely, for each primitive ideal $P$ of $\mathfrak{U}(\l)$, we  exhibit a naturally defined simple $\mathfrak{U}(\l)$-module whose annihilator is precisely $P$. This provides concrete representatives of the simple modules associated with the primitive spectrum of $\mathfrak{U}(\l)$.

We begin by classifying the finite-dimensional simple $\mathfrak{U}(\lambda)$-modules, whose annihilators are the codimension-one primitive ideals.
\begin{corollary}
	A complete set of pairwise non-isomorphic finite-dimensional simple $\mathfrak{U}$-module is
	\begin{equation*}
	\{ [T_{+}(\gamma)] \mid \gamma \in\mK \} \,\, \cup \,\, \{ [T_{-}(\gamma)] \mid  \gamma \in \mK \}, \qquad \text{where } \,\, T_{\pm}(\gamma):=\mathfrak{U}/\langle x,\, g\mp 1, \,y-\gamma \rangle. 
	\end{equation*}
	Moreover, $\ann_\mathfrak{U}(T_{\pm}(\gamma))=\langle x,\, g\mp 1,\, y-\gamma \rangle.$
	In particular, every finite-dimensional simple $\mathfrak{U}$-module is one-dimensional.
\end{corollary}

\begin{proof}
	For each $\gamma\in\mK$, the quotient $T_{\pm}(\gamma)=\mathfrak{U}/\langle x, \,g\mp 1, \,y-\gamma \rangle$ is one-dimensional over $\mK$, and hence is simple. 
	Conversely,  let $M$ be a finite-dimensional simple $\mathfrak{U}$-module. Then $\ann_\mathfrak{U}(M)$ is a primitive ideal of finite codimension. By Theorem \ref{A3Jun26}(3), $\ann_\mathfrak{U}(M)=\langle x, g\mp 1, y-\gamma \rangle$ for some $\gamma \in \mK$. It follows that $M \cong T_{\pm}(\gamma)$. Therefore the modules $T_{\pm}(\gamma)$ exhaust all finite-dimensional
	simple $\mathfrak{U}$-modules. Distinct pairs $(\pm,\gamma)$ yield distinct
	annihilators, and hence non-isomorphic simple modules.
\end{proof}

The following proposition constructs a family of simple $\mathfrak{U}(\l)$-modules whose common annihilator is the primitive ideal $\langle \Omega-\alpha \rangle$. 
\begin{proposition}
	Let $\mathfrak{U}=\mathfrak{U}(\l)$ with $\l \neq 0$. For $\alpha,\beta \in \mK$, define
	\begin{equation*}
	M_{\l}(\alpha, \beta):=\mathfrak{U}/\mathfrak{U}(\Omega-\alpha,\, y+x-\beta g). 
	\end{equation*}
	If $\alpha \neq \pm 4\l$, then $M_{\l}(\alpha, \beta)$ is a simple $\mathfrak{U}$-module, and $\ann_\mathfrak{U}(M_{\l}(\alpha,\beta))=\langle \Omega- \alpha \rangle$. 
\end{proposition}

\begin{proof}
	Recall that $\mathfrak{U}^{\alpha}=\mathfrak{U}/\mathfrak{U}(\Omega-\alpha)$ and $h=\frac{1}{\sqrt{2\l}} yg^{-1}$.
	Since
	$y+x-\beta g=\sqrt{2\lambda}\,g(h-\frac{\beta}{\sqrt{2\lambda}})$,
 and $g$ is invertible in $\mathfrak{U}^\alpha$, it follows that 
	\begin{equation*}
	M_{\l}(\alpha,\beta) \cong \mathfrak{U}^{\alpha}/\mathfrak{U}^{\alpha}(h-\tfrac{\beta}{\sqrt{2\l}}). 
	\end{equation*}
	Assume $\alpha \neq \pm 4\l$. By Proposition \ref{5Jun26}(1), $\mathfrak{U}^{\alpha} \cong R[h;\partial]$ where $R=\mK[e^{\pm 1}]_{\tilde{g}}$. Let $\bar{1}=1+\mathfrak{U}^{\alpha}(h-\tfrac{\beta}{\sqrt{2\l}})$. Then $M_{\l}(\alpha, \beta)=R \bar{1}$. Let $V$ be a nonzero submodule of $M_{\l}(\alpha, \beta)$. Choose a nonzero element $p(e)\tilde{g}^{-m} \bar{1}\in V$ where $0\neq p(e)\in \mK[e^{\pm 1}]$. Since $\tilde g$ is invertible in $R$, we have $p(e)\bar{1}\in V$. Write $p(e)=\sum_{i=r}^s \mu_i e^i$ with $\mu_s \neq 0$. Using $[h, e]=e$, we obtain $he^i \bar{1}=e^i (h+i) \bar{1}=(\tfrac{\beta}{\sqrt{2\l}}+i)e^i \bar{1}$. Thus the vectors $e^i\bar 1$ are $h$-eigenvectors with pairwise distinct eigenvalues. Since $p(e)\bar 1=\sum_{i=r}^{s}\mu_i e^i\bar 1\in V$, 
	and $V$ is stable under the action of $h$, each eigenspace component of $p(e)\bar 1$ belongs to $V$. In particular,
	$e^s\bar 1\in V$.
	Since $e$ is invertible in $R$, it follows that
	$\bar 1=e^{-s}(e^s\bar 1)\in V$.
	Therefore $V=M_{\lambda}(\alpha,\beta)$, and the module $M_{\lambda}(\alpha,\beta)$ is simple.
	
	Clearly,
	$\langle\Omega-\alpha\rangle\subseteq\ann_\mathfrak{U}(M_{\lambda}(\alpha,\beta))$.
	Since $\alpha\neq \pm4\lambda$, Theorem~\ref{A3Jun26}(2) shows that
	$\langle\Omega-\alpha\rangle$
	is a maximal ideal of $\mathfrak{U}$. Therefore
	$\ann_\mathfrak{U}(M_{\lambda}(\alpha,\beta))=\langle\Omega-\alpha\rangle$.
\end{proof}

Let $R$ be an algebra, $M$ a left $R$-module, and $\tau$ an automorphism of $R$. The \emph{twist} of $M$ by $\tau$ is the $R$-module ${}^{\tau}\!M$ defined as follows: as a vector space, ${}^{\tau}\!M=M$, and the $R$-action is given by $r \cdot m=\tau(r)m$, for all $r \in R$ and $m\in M$.  It is immediate that ${}^{\tau}\!M$ is simple if and only if $M$ is simple.   Moreover,  the annihilator of ${}^{\tau}\!M$ is equal to $\tau^{-1}(\ann_R (M))$. If $I$ is a left ideal of $R$, then there is a natural isomorphism ${}^{\tau}\!(R/I) \cong R/\tau^{-1}(I)$. 

The following automorphisms of $\mathfrak{U}(\lambda)$ will be used:
\begin{equation} \label{auto}  
\begin{aligned}
\tau: \qquad &g \mapsto -g, & &x \mapsto x, & &y \mapsto y;\\
\rho: \qquad &g \mapsto g, & &x \mapsto -x, & &y \mapsto -y. 
\end{aligned}
\end{equation} 
Clearly, $\tau^2= \rho^2= \mathrm{id}$. Moreover, $\tau(\phi_{+})=\psi_{-}$, $\tau(\phi_{-})=\psi_{+}$; and $\rho(\phi_{+})=-\phi_{-}$, $\rho(\phi_{-})=-\phi_{+}$. 

The following proposition provides explicit families of simple $\mathfrak{U}(\l)$-modules whose annihilators are precisely the primitive ideals $\langle \phi_{\pm} \rangle$ and $\langle \psi_{\pm} \rangle$. 
\begin{proposition}
	Let $\mathfrak{U}=\mathfrak{U}(\l)$ with $\l \neq 0$, and let $\mu \in \mK^*$ with $\mu \neq 1$. 
	\begin{enumerate}
		\item The $\mathfrak{U}$-modules
		\begin{equation*}
		M_{+}(\mu):=\mathfrak{U}/\mathfrak{U}(\phi_{+}, \, g-\mu), \qquad M_{-}(\mu):=\mathfrak{U}/\mathfrak{U}(\phi_{-}, \, g-\mu),  
		\end{equation*}
		 are simple, with $\ann_\mathfrak{U}(M_{+}(\mu))=\langle \phi_{+} \rangle$ and $\ann_\mathfrak{U}(M_{-}(\mu))=\langle \phi_{-} \rangle$. 
		 
		 \item The $\mathfrak{U}$-modules
		 \begin{equation*}
		 V_{+}(\mu):=\mathfrak{U}/\mathfrak{U}(\psi_{+}, \, g+\mu), \qquad V_{-}(\mu):=\mathfrak{U}/\mathfrak{U}(\psi_{-}, \, g+\mu),  
		 \end{equation*}
		 are simple, with $\ann_\mathfrak{U}(V_{+}(\mu))=\langle \psi_{+} \rangle$ and $\ann_\mathfrak{U}(V_{-}(\mu))=\langle \psi_{-} \rangle$. 
	\end{enumerate}
\end{proposition}

\begin{proof}
	(1) Let $\bar{\mathfrak{U}}=\mathfrak{U}/\langle \phi_{+} \rangle$. By Proposition \ref{2Jun26}, we have $\bar{\mathfrak{U}} \cong \mK[g^{\pm 1}][y;\d_{+}]$ where  $\d_{+}$ is the derivation determined by $\d_{+}(g)=\sqrt{2\l}\,g(g-1)$. Setting $h=\frac{1}{\sqrt{2\l}}yg^{-1}$ gives $[h,g]=g-1$, hence
	\begin{equation} \label{Jphi}  
	\bar{\mathfrak{U}} \cong \mK[g^{\pm 1}][h;\d], \qquad \text{where } \,\d(g)=g-1. 
	\end{equation}
    Since $\phi_{+}$ is normal in $\mathfrak{U}$, it follows that
    \begin{equation*}
    M_{+}(\mu) \cong \bar{\mathfrak{U}}/\bar{\mathfrak{U}}(g-\mu). 
    \end{equation*}
    Let $\bar{1}=1+\bar{\mathfrak{U}}(g-\mu)$. Then, by \eqref{Jphi}, $M_{+}(\mu)=\mK[h]\,\bar{1}$. 
    
    Let $W$ be a nonzero submodule of $M_{+}(\mu)$, and choose a nonzero element $w=\sum_{i=0}^n \alpha_i h^i \bar{1} \in W$ with $n$ minimal and $\alpha_n \neq 0$. Using $(g-1) h^i = (h-1)^i (g-1)$, we obtain  
    \begin{equation*}
    (g-1)w=(\mu-1)\sum_{i=0}^n \alpha_i (h-1)^i \bar{1} \in W. 
    \end{equation*}
    Hence 
    \begin{equation*}
    \begin{aligned}
    w':=(g-1)w-(\mu-1)w =(\mu-1) \sum_{i=0}^n \alpha_i ((h-1)^i-h^i) \bar{1} \in W
        \end{aligned}
    \end{equation*}
    Since $(h-1)^n-h^n=-nh^{\,n-1}+\cdots,$ the element $w'$ has $h$-degree $n-1$. As $\mu\neq 1$, the coefficient of $h^{n-1}\bar 1$ in $w'$ is $-(\mu-1)n\alpha_n\neq 0.$
    This contradicts the minimality of $n$ unless $n=0$. Thus $\bar{1}\in W$, so $W=M_{+}(\mu)$, and $M_{+}(\mu)$ is simple. It follows that $\ann_\mathfrak{U}(M_{+}(\mu))$ is a primitive ideal containing $\langle \phi_{+}\rangle$. Since $\mu\neq 1$, we have $g-1\notin \ann_\mathfrak{U}(M_{+}(\mu))$. By Theorem~\ref{A3Jun26}(3), $\ann_\mathfrak{U}(M_{+}(\mu))=\langle \phi_{+} \rangle$.     
        
    Since $M_{-}(\mu) \cong {}^{\rho}\!M_{+}(\mu)$, it follows that $M_{-}(\mu)$ is simple, and $$\ann_\mathfrak{U}(M_{-}(\mu))=\rho^{-1}(\ann_\mathfrak{U}(M_{+}(\mu))=\langle \phi_{-} \rangle.$$ 
    
    (2) Since $V_{+}(\mu) \cong {}^{\tau}\!M_{-}(\mu)$ and $V_{-}(\mu) \cong {}^{\tau}\!M_{+}(\mu)$, the result follows immediately from part (1). 
\end{proof}

\subsection{The automorphism group}
We now determine the automorphism group of the algebra $\mathfrak{U}=\mathfrak{U}(\l)$ with $\l \neq 0$. Consider the following subgroups of $\Aut(\mathfrak{U})$:
\begin{equation*}
\begin{aligned}
\mathbb{G}&=\{t_{\alpha, \beta}\mid \alpha^2=\beta^2=1\}, \quad &t_{\alpha,\beta}:& \quad g \mapsto \alpha g, \quad x \mapsto \beta x, \quad y \mapsto \beta y. \\
\mathbb{I}&=\{1, \iota \}, \quad &\iota: &\quad g \mapsto g^{-1}, \quad x \mapsto g^{-1}x, \quad y \mapsto -g^{-1}y, \\
\mathbb{A}&=\{ \tau_f \mid f\in \mK[g^{\pm 1},x] \}, \quad &\tau_f:& \quad g \mapsto g, \quad x \mapsto x, \quad y \mapsto y+f. 
\end{aligned}
\end{equation*}
The group $\mathbb{G}$ is isomorphic to $(\Z/2\Z)^2$.  The groups $\mathbb{I}$ and $\mathbb{A}$ coincide with those appearing in the case $\l=0$. 

\begin{theorem}
	Let $\mathfrak{U}=\mathfrak{U}(\l)$ with $\l \neq 0$. Then $\Aut(\mathfrak{U}) = \mathbb{G} \ltimes \mathbb{I} \ltimes \mathbb{A}$. 
\end{theorem}

\begin{proof}
	It is straightforward to verify that $\mathbb{G} \ltimes \mathbb{I} \ltimes \mathbb{A} \subseteq \Aut(\mathfrak{U})$. We prove the reverse inclusion.
	
	Let $\s \in \Aut(\mathfrak{U})$. Since the group of units of $\mathfrak{U}$ is $U(\mathfrak{U})=\{\mK^*g^i \mid i\in \Z \}$, and automorphisms preserve units, we have $\s(g)=\alpha g$ or $\s(g)=\alpha g^{-1}$ for some $\alpha \in \mK^*$. Composing with $\iota$ if necessary, we may assume that $\s(g)=\alpha g$. 
	
	By Theorem \ref{A3Jun26}, the non-maximal height-one prime ideals of $\mathfrak{U}$ are
	\begin{equation*}
	\CO:=\{ \langle \phi_{+} \rangle, \,\, \langle \phi_{-} \rangle, \,\, \langle \psi_{+} \rangle, \,\, \langle \psi_{-} \rangle  \}. 
	\end{equation*} 
	Since automorphisms preserve this set, $\s$ permutes the elements of $\CO$. In particular,  $\s(\langle\phi_{+}\rangle)=\langle \xi \rangle$ for some $\xi \in \{ \phi_{+}, \phi_{-}, \psi_{+}, \psi_{-} \}$. 
	 Since $\phi_{\pm}$ and $\psi_{\pm}$ are normal elements,  Lemma \ref{a11Jun26} implies that $\s(\phi_{+})=u \xi$ for some unit $u\in \mathfrak{U}$. Thus $\s(\phi_{+})=\beta g^i \xi$ for some $\beta \in \mK^*$ and $i\in \Z$. Using the explicit expressions of $\phi_{\pm}$ and $\psi_{\pm}$ in \eqref{phipsi}, we deduce that
	 \begin{equation*}
	 \s(x)=\beta g^i x+p(g)
	 \end{equation*}
	 for some $p(g)\in \mK[g^{\pm 1}]$. Since automorphisms preserve the centre,
	 $\s(\Omega)=\mu \Omega+\nu$ for some $\mu \in \mK^*$ and $\nu \in \mK$. Using $\Omega=x^2g^{-1}-2\l(g^{-1}+g)$, we compute
	\begin{equation*}
	\begin{aligned}
	\s(\Omega) &=\s(x)^2 \s(g)^{-1}-2\l(\s(g)^{-1}+\s(g))\\
	&= \big(\beta g^i x+p(g)\big)^2 \alpha^{-1}g^{-1}-2\l(\alpha^{-1}g^{-1}+\alpha g)\\
	&=\alpha^{-1}\beta^2 g^{2i-1} x^2+2\alpha^{-1}\beta g^{i-1}p(g) x\\
	&\qquad +\alpha^{-1}g^{-1}p(g)^2-2\l(\alpha^{-1}g^{-1}+\alpha g).
	\end{aligned}
	\end{equation*}
	On the other hand,
	\begin{equation*}
	\mu \Omega +\nu = \mu g^{-1} x^2 -2\l \mu(g^{-1}+g)+\nu. 
	\end{equation*}
	Comparing coefficients of $x^2$, $x$, and the remaining terms
	gives
	\begin{equation*}
	 i=0, \quad p(g)=0, \quad\alpha^{-1}\beta^2=\mu, \quad \alpha^{-1}=\mu=\alpha, \quad \nu=0. 
	\end{equation*}
	Hence $\s(x)=\beta x$, and $\alpha^2=\beta^2=1$. 
	
	Replacing $\sigma$ by $t_{\alpha,\beta}^{-1}\sigma$, we may further assume
	that $\sigma(g)=g$ and $\sigma(x)=x$.
	Applying $\sigma$ to the relation $gyg^{-1}=y+x$ yields $g\sigma(y)g^{-1}=\sigma(y)+x.$ Subtracting the original relation gives
	\[
	g\bigl(\sigma(y)-y\bigr)g^{-1}
	=
	\sigma(y)-y.
	\]
	Thus $\sigma(y)-y\in C_\mathfrak{U}(g)$.
	By Lemma~\ref{a10Jun26}, $C_\mathfrak{U}(g)=\Bbbk[g^{\pm1},x],$
	hence $\sigma(y)=y+f$ for some $f\in\Bbbk[g^{\pm1},x]$.
	Therefore
	$\sigma\in \mathbb G\ltimes\mathbb I\ltimes\mathbb A$,
	and thus
	$\Aut(\mathfrak{U})\subseteq \mathbb G\ltimes\mathbb I\ltimes\mathbb A$. Combined with the opposite inclusion, the result follows.
\end{proof}

\subsection{The Dixmier--Moeglin equivalence}
A Noetherian $\mK$-algebra $R$ is said to satisfy the \emph{Dixmier--Moeglin equivalence} if,  for prime ideals of $R$, the following conditions are equivalent: locally closed $\Leftrightarrow$ primitive $\Leftrightarrow$ rational (see \cite[II.8]{Brown-Goodearl-LectQutumGp} for details). Recall that a prime ideal $P$ in a ring $R$ is \emph{locally closed} if the intersection of all prime ideals properly containing $P$ strictly contains $P$. A prime ideal $P$ is called \emph{rational} if the centre of $\Frac(R/P)$ is an algebraic extension of $\mK$. 

The following result provides further evidence for \cite[Conjecture 1.3]{Bell-Leugn}, which predicts that every affine Noetherian Hopf $\mathbb{C}$-algebra of finite Gelfand--Kirillov dimension satisfies the Dixmier--Moeglin equivalence.

\begin{corollary}
	For any $\l \in \mK$,  the algebra $\mathfrak{U}(\l)$ satisfies the Dixmier--Moeglin equivalence. 
\end{corollary}
\begin{proof}
	Since $\mathfrak{U}(\l)$ is a Noetherian algebra satisfying the Nullstellensatz, \cite[Lemma II.7.15]{Brown-Goodearl-LectQutumGp} gives the implications: 
	locally closed  $\Rightarrow$ primitive $\Rightarrow$ rational.	
	By Theorem~\ref{a26May26} (for $\lambda=0$) and Theorem~\ref{A3Jun26} (for $\lambda\neq 0$), the primitive ideals of $\mathfrak{U}(\lambda)$ are exactly the locally closed primes of $\Spec(\mathfrak{U}(\lambda))$. Hence it remains to show that every non-primitive prime is not rational.
	
	If $\lambda=0$, the non-primitive primes are $\langle 0\rangle$, $\langle x\rangle$, and $\langle x,\mathfrak p\rangle$ with $\mathfrak{p} \in  \Spec( \mK[g^{\pm1},y])$ of height one. Since $\Omega$ is transcendental over $\mK$ and $\mK(\Omega)\subset Z(\Frac(\mathfrak{U}(0)))$, the zero ideal is not rational. For any non-maximal prime $P\supseteq \langle x\rangle$,  the quotient $\mathfrak{U}(0)/P$ is a commutative domain of positive transcendence degree over $\mK$, hence $Z(\Frac(\mathfrak{U}(0)/P))=\Frac(\mathfrak{U}(0)/P)$ is not algebraic over $\mK$, so $P$ is not rational.
	
	If $\lambda\neq 0$, the only non-primitive primes are $\langle x,g-1\rangle$ and $\langle x,g+1\rangle$. For either such prime ideal $P$, we have $\mathfrak{U}/P \cong \mK[y]$. Hence $Z(\Frac(\mathfrak{U}/P))=\mK(y)$, which is transcendental over $\mK$.

	Thus locally closed, primitive, and rational primes coincide, so $\mathfrak{U}(\lambda)$ satisfies the Dixmier--Moeglin equivalence.
\end{proof}

\

\

\noindent
\textbf{Data availability statement}  Data sharing is not applicable to this article as no datasets were generated or analysed during the current study.

\

\noindent
\textbf{Declarations} The author declares no conflict of interest.

\small{

\end{document}